\documentclass[11pt]{amsart}
\usepackage{preamble}

\begin{document}

\title[On finiteness of log canonical models]{On finiteness of log canonical models}


\begin{abstract}
Let $(X, \De)/U$ be klt pairs and $Q$ be a convex set of divisors. Assuming that the relative Kodaira dimensions are non-negative, then there are only finitely many log canonical models when the boundary divisors varying in a relatively compact rational polytope in $Q$. As a consequence, we show the existence of the log canonical model for a klt pair $(X, \De)/U$ with real coefficients.
\end{abstract}

\author{Zhan Li}
\address{Department of Mathematics, Southern University of Science and Technology, 1088 Xueyuan Rd, Shenzhen 518055, China} \email{lizhan@sustech.edu.cn}

\maketitle

\setcounter{tocdepth}{1}
\tableofcontents

\section{Introduction}\label{sec: introduction}

Throughout this paper, we work with varieties defined over complex numbers.

A minimal model of a variety is a good representative in its birational equivalent class which could substitute the original variety in many situations. However, minimal models lack the uniqueness. To remedy this problem, one can pass to its log canonical model which is unique when it exists. 

Let $(X, \De)$ be a projective klt pair and $K_X+\De$ is $\Qq$-Cartier. Combing with the method of \cite{FM00} and \cite[Theorem 1.2]{BCHM10} (see \cite{Siu06} for an analytic proof), \cite[Corollary 1.1.2]{BCHM10} establishes that the log canonical ring
\begin{equation}\label{eq: canonical ring}
R(X, K_X+\Delta)=\bigoplus_{m=0}^\infty H^0(X, \Oo_{X}(\lf m(K_X+\Delta)\rf))
\end{equation} is finitely generated. In particular, \[
 X \dasharrow \proj R(K_X+\Delta)
 \] is the log canonical model for $(X, \De)$. This shows that $(X, \De)$ always admits the log canonical model even if a (good) minimal model of $(X, \De)$ is not assumed to exist. 

However, when $K_X+\De$ is an $\Rr$-Cartier divisor, the log canonical model may still exist while $R(X, K_X+\Delta)$ is not finitely generated. The motivation of this paper is to study the log canonical model in such situation.  Along the way, we establish the finiteness of log canonical models assuming that the relative Kodaira dimensions are non-negative. The precise statement is the following. 

Notice that for a rational map $\phi: X \dto Z/U$ and a rational polytope $P \subset \WDiv(X)_{\Rr}$, $\Aa_{\phi, \pi}(P)$ is defined to be
\[
\{\De \in P \mid \phi \text{~is the log canonical model for~} (X, \De) \text{~over~} U\}.
\] For other relevant definitions and notation, see Section \ref{sec: preliminaries}. 

\begin{theorem}\label{thm: main}
Let $\pi: X\to U$ be a projective morphism between normal varieties and $Q \subset \WDiv(X)_{\Rr}$ be a convex set. Suppose that for any $\De \in Q$, $(X, \De)$ is klt with $\ka(K_X+\De; X/U) \geq 0$. Let $P$ be a rational polytope which is contained in the relative interior of $Q$. Then there are finitely many rational maps $\phi_i: X \dto Z_i/U, 1 \leq i \leq q$, such that $P$ is partitioned into subset $\Aa_i=\Aa_{\phi_i,\pi}(P)$, that is, 
\[
P=\bigcup_{i=1}^q \Aa_i.
\] Besides, the closure $\bar\Aa_i$  of each $\Aa_i$ is a finite union of rational polytopes.
\end{theorem}

This result is an analogy of \cite[Corollary 1.1.5]{BCHM10} where $\De \in A+V$ with $A$ a general ample divisor and $V$ a finite dimensional affine subspace of divisors. See Remark \ref{rmk: relatively compact} for the reason to work with the relatively compact subset $P$ and Remark \ref{rmk: constant Iitaka is enough} for a variant of Theorem \ref{thm: main}. Theorem \ref{thm: main} is indeed about finiteness of \emph{marked} log canonical models, that is, not only the variety $Z_i$ itself is finite, but also the map $X \dto Z_i/U$ is finite. Although finiteness of minimal models are expected for a given klt pair, finiteness of marked minimal models may not hold true in general.

For a klt pair $(X, \De)$ over $U$, we assume that $K_X+\De$ is $\Rr$-Cartier with $X \to U$ a projective morphism. As a consequence of Theorem \ref{thm: main}, we can answer the aforementioned problem in the relative setting.

\begin{corollary}\label{cor: canonical model for R-div}
Suppose that $(X, \De)$ is a klt pair over $U$ with $\ka(K_X+\De/U) \geq 0$, then $(X, \De)$ has a log canonical model over $U$.
\end{corollary}

The type of results as Theorem \ref{thm: main} was studied by Shokurov which was called the geography of log models (see \cite[\S 6]{Sho96}). A stronger form of Theorem \ref{thm: main} has already been established in \cite{SC11} assuming the log minimal model program and the abundance conjecture. By \cite[Theorem 3.4]{SC11}, a rational polytope can be decomposed into a finite union of rational polytopes by the equivalence of the weak log canonical models. Then the abundance conjecture implies that the log canonical models are also finite. In fact, \cite{SC11} uses the above decomposition to show that the abundance conjecture for $\Rr$-divisors follows from that for $\Qq$-divisors (see \cite[Corollary 2.2]{SC11})

We briefly explain the idea for the proof of Theorem \ref{thm: main}. First, we consider the relative Iitaka fibration (see Definition \ref{def: Iitaka fibration}) associated with $K_X+\De$ for $\De \in Q$. The key observation is that an Iitaka fibration $X \dto Z$ can be taken uniformly for any $\De \in P \subset Q^\circ$ (see Proposition \ref{prop: uniform Iitaka fibration}). Then, applying the canonical bundle formula in \cite{FM00}, the log canonical model for $(X, \De)$ can be related to the ample model for $(Z, D_Z^\De+M_Z^\De)$. Finally, by standard modifications, the result follows from \cite[Corollary 1.1.5]{BCHM10}.

The paper is organized as follows. In Section \ref{sec: preliminaries}, we give background material. In Section \ref{sec: relative Kodaira dim}, we develop relative Iitaka fibrations and relative Kodaira dimensions. Most results in this section are standard and should be well-known. Unfortunately, we find it lack of adequate reference. Hence, we include a systematical treatment of this topic in Section \ref{sec: relative Kodaira dim}. Theorem \ref{thm: main} and Corollary \ref{cor: canonical model for R-div} are proved in Section \ref{sec: finite ample models}. 

\medskip

\noindent\textbf{Acknowledgements}.
This work is motivated by expanding the results in \cite{Li20} to real divisors. When preparing the paper, it appears \cite{Jia20} which also obtains some of the results in the paper independently. This work is partially supported by a starting grant from SUSTech.

\section{Preliminaries}\label{sec: preliminaries}

\subsection{Notation and conventions}\label{subsec: Notation and conventions}
By polytope, we mean a convex hull of finite set of points in a real vector space. If the points are rational points, then the polytope is called rational. In particular, a polytope is assumed to be closed. 

For a birational morphism $f: Y \to X$ and a divisor $B$ on $X$, $f_*^{-1}(B)$ denotes the strict transform of $B$ on $Y$, and $\Exc(f)$ denotes the sum of reduced exceptional divisors of $f$. We write $f: Y \to X/U$ if $f$ is a morphism over $U$. We say that $f$ is a contraction morphism if $f_*\Oo_Y = \Oo_X$. If $f$ is a surjective morphism, then for a divisor $D$ on $Y$, we write $D^h$ (resp. $D^v$) to denote the horizontal (resp. vertical) part of $D$ over $Z$. We write $D=D_+- D_-$ for the decomposition of $D$ as a difference of effective divisors without common components.

For $k=\Zz, \Qq, \Rr$, and two divisors $A, B \in k$ on a variety $X$ over $Z$, $A \sim_{Z,k} B$ means that $A$ and $B$ are $k$-linearly equivalent over $Z$. When $k=\Zz$ or $Z=\spec \Cc$, we omit $k$ or $Z$. A general (resp. very general) point is a point which lies in the complement of finite (resp. countably many) proper Zariski closed subsets.

Let $X$ be a projective normal variety over $U$ and $B$ be an $\Rr$-divisor on $X$, then $(X, B)/U$ is called a log pair over $U$. We also assume that $K_X+B$ is an $\Rr$-Cartier divisor for a log pair $(X, B)/U$. Besides, $U$ is usually omitted if this is clear from the context. For a divisor $D$ over $X$, if $f: Y \to X$ is a birational morphism from a normal variety $Y$ such that $D$ is a divisor on $Y$, then the log discrepancy of $D$ with respect to $(X, B)$ is defined to be $\mult_{D}(K_Y-f^*(K_X+B))+1$. This definition is independent of the choice of $Y$. A log pair $(X, B)$ (or $K_X+B$) is called sub-klt (resp. sub-lc) if the log discrepancy of any divisor over $X$ is $>0$ (resp. $\geq 0$). If $B \geq 0$, then a sub-klt (resp. sub-lc) pair $(X, B)$ is called klt (resp. lc). 

\subsection{Ample model and log canonical model}\label{subsec: ample model}

We adopt the notions of ample models and log canonical models as in \cite[Definition 3.6.5, 3.6.7]{BCHM10}.

\begin{definition}[{Ample model \& log canonical model }]\label{def: ample and canonical model}
Let $\pi: X\to U$ be a projective morphism between normal quasi-projective varieties and $D$ be an $\Rr$-Cartier divisor on $X$. We say that $g: X \dto Z$ is the ample model for $D$ over $U$ if $g$ is a rational map over $U$, $Z$ is normal and projective over $U$ and there is an ample divisor $H$ over $U$ on $Z$ such that if $p: W \to X$ and $q: W \to Z$ resolve $g$, then $q$ is a contraction morphism and we may write $p^*D \sim_{U, \Rr} q^*H+E$, where $E \geq 0$ and for every $B \in |p^*D/U|_{\Rr}$, then $B \geq E$. If $D=K_X+\De$, where $(X, \De)$ is an lc pair, then the ample model for $K_X+\De$ over $U$ is called the log canonical model for $K_X+\De$ over $U$.
\end{definition}

\begin{remark}
1. By definition, for any $D'$ and $\lambda \in \Rr_{>0}$ such that $D' \sim_{U, \Rr} \lambda D$, $D'$ and $D$ share the same ample model if one of them exists.

2. For a given $X/U$, ample models of $D$ over $U$ are unique up to isomorphism (see \cite[Lemma 3.6.6 (1)]{BCHM10}).
\end{remark}

\subsection{Canonical bundle formula}\label{subsec: canonical bundle formula}

We use the relative version of the following canonical bundle formula proved in \cite{FM00}. 

Let $f: Y \to Z/U$ be a surjective and projective morphism of a normal variety $Y$ to a nonsingular variety $Z$. Assume that $Z$ is projective over $U$ and $(Y, \De)$ is klt with $K_Y+\De$ a $\Qq$-Cartier divisor. Suppose that the generic fiber $F$ of $f$ is a geometrically irreducible variety with $\ka((K_Y+\De)|_F)=0$. By \cite[Proposition 4.2]{FM00}, there exist $b \in \Zz_{>0}$ and a $\Qq$-divisor $\Theta$ such that
 \begin{equation}\label{eq: canonical bundle relation}
 b(K_Y+\De) \sim f^*(b(K_Z+\Theta))+bB^\De,
 \end{equation} where $B^\De= B^\De_+-B^\De_-$ satisfies that $f_*\Oo_Y(\lf kB^\De_+ \rf)=\Oo_Z$ for any $k \in \Zz_{\geq 0}$ and $\codim f(B^\De_-) \geq 2$. Notice that $K_Z+\Theta$ is unique up to $\Qq$-linearly equivalence. Moreover, for a prime divisor $P$ on $Z$, let $t_P^\De$ be the log canonical threshold of $f^*P$ with respect to $(X, \De - B^\De)$ over the generic point $\eta_P$ of $P$. Set $s_P^\De = 1- t_P^\De$, then 
 \begin{equation}\label{eq: divisorial part}
 D^\De_{Z} \coloneqq \sum_P s_P^\De P
 \end{equation} is called the divisorial part, and 
 \begin{equation}\label{eq: moduli part}
 M^\De_{Z} \coloneqq \Theta -  D_{Z} 
 \end{equation} is called the moduli part (it is called the log-semistable part in  \cite[Definition 4.3]{FM00}). In fact, by the discussion before \cite[Theorem 2.7]{Amb04} and  \cite[Proposition 4.6]{FM00}, there are b-divisors $\mathbf D^\De$ and $\mathbf M^\De$ such that $D^\De_{Z}$ and $M^\De_{Z}$ are traces of the respective b-divisors on $Z$ (see  \cite[\S 1.2]{Amb04} or \cite[\S 2.3]{Cor07} for notions of b-divisors.)

 After taking certain model $f': Y' \to Z'/U$ birational to $f: Y \to Z/U$ (see \cite[\S 4.4]{FM00} for the precise definition of $f'$), the corresponding moduli part $M^\De_{Z'}$ is nef over $U$. In fact, \cite[Lemma 5.2 (5)]{Amb04} shows that $M^\De_{Z'}$ is semi-positive. Hence, for any complete curve $C$, $M^\De_{Z'} \cdot C \geq 0$. Because $Z'$ is projective over $U$, this implies that $M^\De_{Z'}$ is nef over $U$. Moreover, $\mathbf M^\De$ is b-nef/$U$ (see \cite[Theorem 2.7]{Amb04}).
 
 \begin{lemma}\label{le: redundant part equal}
Assume that $f: Y \to Z$ is a projective morphism and 
 \[
 f^*L + B \sim f^*G+ C
 \] with $L, G$ Cartier divisors.  Suppose that $B, C$ satisfy the property that  $B=B_+-B_-, C=C_+-C_-$ with $f_*\Oo_Y(\lf kB_+ \rf)=f_*\Oo_Y(\lf kC_+ \rf)=\Oo_Z$ for any $k \in \Zz_{\geq 0}$, and $\codim f(B_-) \geq 2, \codim f(C_-) \geq 2$. Then $B= C$.
 \end{lemma}
 \begin{proof}
 By $f_*\Oo_Y(f^*L + B) = f_*\Oo_Y(f^*G + C)$ and $$f_*\Oo_Y(\lf kB_+ \rf)=f_*\Oo_Y(\lf kC_+ \rf)=\Oo_Z,$$ we have $L \sim G$ on $Z \backslash \Supp(f(B_-)\cup f(C_-))$. This implies $L \sim G$ on $Z$ by $\codim \Supp(f(B_-)\cup f(C_-)) \geq 2$. Hence by definition, there is $s \in K(Y)$ such that
 \[
 B + {\rm div}(s)= C .
 \] Thus $B_+-C_++ {\rm div}(s)=B_--C_-$, and $B_+|_F+{\rm div}(s)|_F=C_+   \geq 0$ for a very general fiber $F$. However, $f_*\Oo_Y(\lf k B_+\rf)=\Oo_Z$ for all $k \in \Zz_{\geq 0}$ implies that $H^0(F, \Oo_F(\lf k B_+\rf)=\Cc$. Thus $s$ must be a non-zero constant on a very general fiber. Hence $B^h=C^h$. By $f_*\Oo_Y(\lf k B_+\rf)=f_*\Oo_Y(\lf k C_+\rf)=\Oo_Z$ for all $k \in \Zz_{\geq 0}$, $B_+^v$ and $C_+^v$ do not have components whose images on $Z$ are of codimension $1$. Thus, by assumption, $\codim f(B^v_-+B^v_+) \geq 2, \codim f(C^v_-+C^v_+) \geq 2$. Because $B^v \sim C^v$, we have $B^v = C^v$.
 \end{proof}

\section{Relative Kodaira dimensions}\label{sec: relative Kodaira dim}

In this section, we assume that $\pi: X\to U$ is a projective morphism between varieties. Assume that $X$ is normal and $D$ is an $\Rr$-Weil divisor on $X$. If there exists $m \in \Zz_{>0}$ such that $\pi_*\Oo_X(\lfloor mD \rfloor) \neq 0$, then there is a rational map over $U$
\[
\phi_{U, |mD|}: X \dto \Pp_U(\pi_*\Oo_X(\lfloor mD \rfloor))
\] associated with $\pi^* \pi_*\Oo_X(\lfloor mD \rfloor) \to \Oo_X(\lfloor mD \rfloor)$, where $$ \Pp_U(\pi_*\Oo_X(\lfloor mD \rfloor)) \coloneqq {\underline{\proj}_U} ({\rm Sym}(\pi_*\Oo_X(\lfloor mD \rfloor))).$$ 

We say that $m$ is sufficiently large and divisible if $\phi_{U, |mD|}(X)$ achieves its maximal dimension and general fibers of $\phi_{U, |mD|}$ are connected. In fact, take $m$ such that $\dim \phi_{U, |mD|}(X)$ is maximal and replace $X$ by a resolution, we can assume that $\phi_{U, |mD|}$ is a morphism. Replacing $m$ by a multiple $km$ so that $\tau^*\Oo_{W_m}(k)$ is very ample on the Stein factorization $X \to W_{m, {\rm stein}} \xrightarrow{\tau} W_m$, we can assume that $\phi_{U, |mD|}$ has connected general fibers. In this case, because $k(X)$ is a regular extension of $k(W_m)$,  the generic fiber of $\phi_{U, |mD|}$ is geometrically integral (for example, see \cite[\S 3 Corollary 2.14 (c)]{Liu02}).

Let $u: X' \to X$ be a log resolution and set $D' =u_*^{-1}D+E$ where $E$ is a $u$-exceptional divisor with sufficiently large coefficients (for example, take a Cartier divisor $B>0$ on $X$ such that $B - D >0$, then $E$ can be taken to be $u^*B-u_*^{-1}B$). Then $u_*\Oo_{X'}(\lfloor mD'\rfloor) = \Oo_X(\lfloor mD \rfloor)$ for any $m \in \Zz_{\geq 0}$. Thus $\phi_{U, |mD'|}: X' \dto \Pp_U((\pi\circ u)_*\Oo_{X'}(\lfloor mD' \rfloor))$ has the same image as $\phi_{U, |mD|}$. In what follows, when taking a resolution of $X$, we always replace $D$ by $D'$ (for some $E$). Possible taking a further resolution, we can assume that $|\lfloor mD'\rfloor| = |P_m|+N_m$ such that $N_m$ is effective, $|P_m|$ is $\pi\circ u$-base point free and $\pi_*\Oo_X(\lfloor mD \rfloor) =(\pi\circ u)_*\Oo_{X'}(P_m)$. Then the morphism $f': X' \to \Pp_U(\pi_*\Oo_X(\lfloor mD \rfloor))$ associated with $(\pi\circ u)^*(\pi\circ u)_*\Oo_{X'}(P_m) \to \Oo_{X'}(P_m)$ resolves $\phi_{U, |mD|}$. Let $W_m \subset \Pp_U(\pi_*\Oo_X(\lfloor mD \rfloor))$ be the image of $X'$, then there is a divisor $A$ on $W_m$ which is very ample over $U$ such that $P_m=f'^*A$. For $\Rr$-divisors $D_1, D_2$, we write 
\begin{equation}\label{eq: gtrsim}
D_1 \gtrsim_{U, \Rr} D_2
\end{equation} if there exists an effective divisor $E \geq 0$ such that $D_1 \sim_{U, \Rr} D_2+E$. Hence the above shows that $\lf mD' \rf \gtrsim_{U, \Rr} f'^*A$.

\begin{definition}[{Relative Kodaira dimension, c.f.\cite[II. \S 3.c]{Nak04}}]\label{def: kod}
Let $\pi: X\to U$ be a projective morphism between normal varieties and $D$ be an $\Rr$-Weil divisor on $X$, then the relative Kodaira dimension is defined to be 
\[
\ka(D; X/U)\coloneqq 
\begin{cases}
  -\infty   \quad\quad\quad \text{if $\pi_*\Oo_X(\lfloor mD \rfloor)=0$ for any~}  m \in \Zz_{>0}, \\
  \max\{\dim \phi_{U, |mD|}(X) \mid m \in \Zz_{\geq 0}\} - \dim U ~  \text{otherwise}.
\end{cases}
\]
\end{definition}

When $U = \spec \Cc$, we write $\ka(D;X)$ or $\ka(D)$ instead of $\ka(D;X/\spec \Cc)$. Because $\pi_*\Oo_X(\lfloor mD \rfloor)$ is a torsion free sheaf and $\phi_{U, |mD|}(X)$ is irreducible if it is non-empty, by the above definition, for any non-empty open set $V \subset U$, we have $\ka(D; X/U) = \ka(D_V; X_V/V)$, where $X_V = X \times_U V$, and $D_V= D|_{X_V}$. 

\begin{remark}[{\cite[\S 2.2 Remark]{Cho08}}]\label{rmk: choi}
For $\Rr$-divisors $D$ and $B$ on $X$. Even if $D \sim_{U, \Rr} B$, the relative Kodaira dimensions may not be the same: consider a smooth projective variety $X$ over $\Cc$. Let $P$ be a non-trivial principle Cartier divisor. Set $D = a P$ for some $a \in \Rr \backslash \Qq$. Then $D \sim_{\Rr} B=0$. By $H^0(X, \Oo_X(\lf mD \rf)) = 0$ for any $m \in \Zz_{>0}$, we have $\ka(D) = 0$ which is not equal to $\ka(B) =1$. 
\end{remark}

To remedy the problem mentioned in Remark \ref{rmk: choi}, \cite[Definition 2.2.1]{Cho08} introduces the invariant Iitaka dimension. In particular, the following result holds.

\begin{proposition}[{\cite[Proposition 2.2.2]{Cho08}}]\label{prop: Choi}
If $\ka(D; X/U) \geq 0$, then for any $0 \leq D' \sim_{U, \Rr} D$, we have $\ka(D; X/U) = \ka(D'; X/U)$.
\end{proposition}

Suppose that $\pi: X\to U$ is a projective morphism between normal varieties and $D$ is an $\Rr$-Weil divisor on $X$.  Let $F$ be a general fiber of $\pi$. We define the Kodaira dimension $\kappa(D|_F)$ of $D|_F$ to be $\ka(D|_{\Gamma})$, where $\Gamma$ is an irreducible component of $F$. Equivalently, suppose that $X \to U_{\rm stein} \xrightarrow{\tau} U$ is the Stein factorization and $G$ is a general fiber of $X \to U_{\rm stein}$. Then $\kappa(D|_F)= \kappa(D|_G)$. In fact, because $U_{\rm stein} \to U$ is a finite morphism between irreducible varieties, for any open set $V_s \subset U_{\rm stein}$, there is an open set $V \subset U$ such that $V_s \supset \tau^{-1}(V)$. Hence any irreducible component of $F$ is also a general fiber of $X \to U_{\rm stein}$. Thus $\ka(D|_F)$ is well defined and equals to $\ka(D|_G)$.

We will use the fact that the set of very general points is a complement of a measure zero set (because the base field is $\Cc$). Hence if $V$ and $V'$ are sets consist of very general points, then $V \cap V' \neq \emptyset$. The following is the relative version of \cite[II Lemma 1.12 (2)]{Nak04} which can be proved by the same argument.

\begin{lemma}[{\cite[II Lemma 1.12 (2)]{Nak04}}]\label{le: contract to point}
Suppose that $Y, X$ are normal varieties over $U$. Let $f: Y \dto X/U$ be a meromorphic fiber space and let $h: Y \dto Z/U$ be a meromorphic map such that $h(f^{-1}(x))$ is a point for general $x \in X$. Then there exists a meromorphic map $g: X \dto Z/U$ such that $h=g \circ f$.
\end{lemma}

In the above lemma, $f$ is called a meromorphic fiber space if $Y, X$ are normal varieties, the graph map $\Gamma_f \to X/U$ is a proper morphism and any general fiber of $\Gamma_f^\nu \to X$ is connected, where $\Gamma_f^\nu \to \Gamma_f$ is the normalization. 

\begin{lemma}\label{le: compare images of iitaka fibration} 
Let $\pi: X\to U$ be a projective morphism between normal varieties and $D$ be an $\Rr$-Weil divisor on $X$. Suppose that $$\phi_{U, |mD|}: X \dto W_m \subset \Pp_U(\pi_*\Oo_X(\lfloor mD \rfloor)) $$ is a rational map. Assume that general fibers $F$ of $\phi_{U, |mD|}$ are connected and $\ka(D|_F)=0$. Then for any sufficiently large and divisible $k\in \Zz_{>0}$, there is a birational map $\tau_{m,k}: W_m \dto W_{k}/U$ such that $ \tau_{m,k} \circ\phi_{U, |mD|} = \phi_{U, |kD|}$.
\end{lemma}
\begin{proof}
First, replace $X$ by a resolution such that $p: X \to W_m/U, q: X \to W_{mk}/U$ and $r: X \to W_{k}/U$ are all morphisms. For a general fiber $F_p$ of $p$, $$H^0(F_p, \Oo_{F_p}(\lf mD \rf|_{F_p}))=\Cc$$ by construction, and thus $H^0(F_p, \Oo_{F_p}(\lf mkD \rf|_{F_p}))=\Cc$ by $k\lf mD \rf \leq \lf mkD \rf$ and $\ka(D|_{F_p})=0$. Thus $q(F_p)$ is a point of $W_{mk}$. By Lemma \ref{le: contract to point}, there exists $\psi: W_m \dto W_{mk}/U$ such that $\psi \circ p =q$. In particular, $\psi$ is dominant and has generic connected fibers. By the choice of $k$, $\dim W_m \leq \dim W_{mk}$, and thus  $\dim W_m = \dim W_{mk}$. This implies that $\psi$ is a birational map. By the same argument, we have a birational map $\varphi: W_k \dto W_{mk}/U$ such that $\varphi \circ r =q$. Thus $\tau_{m,k} \coloneqq \varphi^{-1} \circ \psi: W_m \dto W_{k}/U$ is birational and $ \tau_{m,k} \circ\phi_{U, |mD|} = \phi_{U, |kD|}$.
\end{proof}

\begin{lemma}\label{le: general fiber has Kod=0}
Under the above notation, suppose that $\pi_*\Oo_X(\lf aD \rf) \neq 0$ for some $a \in \Zz_{> 0}$. For a sufficiently large and divisible $n$, assume that $\phi': X' \to W_n/U$ is a resolution of the map $\phi=\phi_{U, |nD|}: X \dto W_n/U$. Then for any general fiber $G'$ of $\phi'$, we have $\ka(D'|_{G'})=0$, where $D'$ is the strict transform of $D$ plus an exceptional divisor over $X$ with sufficiently large coefficients (see the third paragraph in Section \ref{sec: relative Kodaira dim}).
\end{lemma}
\begin{proof}
Without loss of generality, we can replace $X'$ by $X$ and $D'$ by $D$.  Let $A$ be a very ample divisor on $W_n$ over $U$ such that $\lf n D \rf \gtrsim_{U, \Rr}  \phi^*A$ (see \eqref{eq: gtrsim} for ``$\gtrsim_{U, \Rr} $'').  By construction, $\ka(D|_{G}) \geq 0$. Let $\emptyset \neq V_m \subset W_n$ be an open set such that $h^0(G_t, \Oo_{G_t}(\lf mD|_{G_t} \rf))$ is a constant for $t \in V_m$. By $\ka(D|_{G_t}) \geq 0$, for a sufficiently large and divisible $m$, $h^0(G_t, \Oo_{G_t}(\lf mD|_{G_t} \rf))  \geq 1$ for $t \in V_m$. However, if there exists $m$ such that $h^0(G_t, \Oo_{G_t}(\lf mD|_{G_t} \rf))  > 1$, we claim that there exists $\nu \gg 1$ such that $\dim W_{\nu} > \dim W_n$. 

By shrinking $V_m$ and $U$, there is an isomorphism $$\phi_*\Oo_X(\lf mD \rf) \otimes k(t) \simeq H^0(G_t,  \Oo_{G_t}(\lf mD|_{G_t}\rf))$$ for each $t\in V_m$ (see \cite[III Corollary 12.9]{Har77}). Such morphism is obtained by sending $$s \in H^0(\phi^{-1}(V_m), \Oo_X(\lf mD \rf))$$ to $s|_{G_t}$. By $h^0(G_t, \Oo_{G_t}(\lf mD|_{G_t} \rf))  > 1$, there are at least two linearly independent sections $s_1, s_2 \in H^0(\phi^{-1}(V_m), \Oo_X(\lf m D \rf))$. Thus $\tilde \phi = \phi_{V_m, \lf mD \rf}: X \dto T_m/V_m$ with $\dim T_m>\dim V_m$. 

Taking a higher resolution of $X$, we can assume that $\tilde \phi $ is a morphism and there is a divisor $H$ which is very ample on $T_m$ over $V_m$ such that $\lf mD \rf \gtrsim_{T_m, \Rr}  {\tilde\phi}^*H$. Let $\mu: T_m \to V_m$ and choose $\ell \gg 1$ such that $H +\ell\mu^*A$ is ample over $U$. Then $\lf (m+\ell n) D \rf \gtrsim_{U, \Rr}  \varphi^*H+\ell\phi^*A$ and thus there exists $\tau \gg 1$ such that $\dim \phi_{U, |\tau(m+\ell n) D|}(X) \geq \dim T_m$. Hence $$\dim W_{\tau(m+\ell n)} \geq \dim T_m > \dim V_m =\dim W_m,$$ and we can take $\nu=\tau(m+\ell n)$. This is a contradiction to the maximality of $\dim W_m$, and thus $\ka(D|_{G})= 0$.
\end{proof}

\begin{lemma}\label{le: very general implies general in kod k}
Under the above notation, for $k \in \Zz_{\geq 0}$, if $\ka(D|_{F_t})=k$ for very general $t\in U$, then $\ka(D|_{F_t})=k$ for general $t\in U$.
\end{lemma}
\begin{proof}
By $k \geq 0$, we claim that there exists $m \in \Zz_{>0}$ such that $\pi_*\Oo_X(\lf mD \rf) \neq 0$. Otherwise, as above, there is an open set $U_m \subset U$ such that $$H^0(F_t,  \Oo_{F_t}(\lf mD|_{F_t}\rf)) \simeq \pi_*\Oo_X(\lf mD \rf) \otimes k(t) =0$$ for each $t \in U_m$. Thus for $t\in \cap_m U_m$, we have $\ka(D|_{F_t}) = -\infty$, a contradiction. 

Hence, we can consider $\phi=\phi_{U, |nD|}: X \dto W_n/U$ for a sufficiently large and divisible $n$. Taking a resolution of $X$, we can assume that $\phi$ is a morphism. By Lemma \ref{le: general fiber has Kod=0}, $\ka(D|_{G})=0$ for any general fiber $G$ of $\phi$.

Now for a general fiber $F_t$ of $\pi$, let $m_t$ be a sufficiently large and divisible integer such that $\phi_{|m_tD|_{F_t}|}: F_t \dto W_{t, m_t}$ is the corresponding rational map. For $\phi|_{F_{t}}: F_t \dto \phi|_{F_{t}}(F_t)$, a general fiber $G$ is also a general fiber of $\phi$ and thus $\ka(D|_G)=0$. By Lemma \ref{le: compare images of iitaka fibration}, $W_{t, m_t}$ is birational to $\phi|_{F_{t}}(F_t)$ and thus $$\ka(D|_{F_t})=\dim W_{t, m_t} = \dim \phi|_{F_{t}}(F_t)=\dim W_n-\dim U$$ for general $t \in U$. Finally, because $\ka(D|_{F_t})=k$ for very general $t$, $k=\dim W_n-\dim U$.
\end{proof}

\begin{remark}
I do not know whether Lemma \ref{le: very general implies general in kod k} still holds true when $k=-\infty$. Let $Z_m \subset U$ be the closed set such that $h^0(F_t, \Oo_{F_t}(\lf mD|_{F_t}\rf))>0$ for $t\in Z_m$. Then $Z = \cup_m Z_m$ is the set of points $t$ such that $\ka(D|_{F_t}) \neq -\infty$. A priori, it is possible that the Zariski closure of $Z$ is $X$.

However, in the most important case when $D = K_X$ on a smooth variety. By the invariance of the plurigenera (\cite{Siu02}), Lemma \ref{le: very general implies general in kod k} still holds true when $k=-\infty$.  
\end{remark}

\begin{lemma}\label{le: fiber kad positive}
Under the above notation, for a general fiber $F$ of $\pi$, if $\ka(D|_F)>0$, then $\ka(D;X/U)>0$.
\end{lemma}
\begin{proof}
We claim that there exist $n \in \Zz_{\geq 0}$ and an open set $U_n \subset U$ such that $h^0({F_t}, \Oo_{F_t}(\lf n D\rf|_{F_t}))>1$ for any $t\in U_n$. 

Taking a resolution of $X$ and shrinking $U$, we can assume that $(X, D)$ is log smooth over $U$ (i.e. $X$ as well as any stratum of $D$ is smooth over $U$). By upper-semicontinuity, there is an open set $U_m$ for each $m$, such that $$h^0(F_t, \Oo_{F_t}(\lf m D|_{F_t}\rf))=h^0({F_t}, \Oo_{F_t}(\lf m D\rf|_{F_t}))=k_m$$ is a constant for $t\in U_m$. If $k_m=0$ for $m \in \Zz_{>0}$, then for $t\in \cap_m U_m$, $\ka(D|_{F_t}) = -\infty$. This is a contradiction to $\ka(D|_F)>0$ for a general fiber $F$. If $k_m \leq 1$ for any sufficiently large and divisible $m \in J$, then for $t\in \cap_{m \in J} U_m$, $\ka(D|_{F_t}) = 0$, which is still a contradiction. Hence, there exists $n$ such that $k_n>1$. This shows the claim.

Possibly shrinking $U_n$, we can assume that $$\pi_*\Oo_X(\lf nD \rf) \otimes k(t) \simeq H^0(F_t,  \Oo_{F_t}(\lf nD|_{F_t}\rf))$$ holds for each $t\in U_n$. Then just as in the proof of Lemma \ref{le: very general implies general in kod k}, sections of $H^0(F_t,  \Oo_{F_t}(\lf nD|_{F_t}\rf))$ extend to local sections of $\pi_*\Oo_X(\lf n D\rf)$. By $h^0(F, \lf n D|_F\rf)>1$, the image of $\phi_{U, |nD|}$ has dimension larger than $\dim U$, that is, $\ka(D;X/U)>0$.
\end{proof}

\begin{definition}[{Relative $D$-dimension  \cite[Chapter II \S 3.c]{Nak04}}]\label{def: relative D-dim}
Let $\pi: X \to U$ be a projective morphism between varieties and $D$ be an $\Rr$-Weil divisor on $X$. The relative $D$-dimension $\ka_{\rm rel}(D; X/U)$ is defined to be the Kodaira dimension $\ka(D|_\Gamma)$ for a connected component  $\Gamma$ of a very general fiber of $\pi$.
\end{definition}

\begin{remark}\label{rmk: fibers}
1. A fiber of $\pi$ may not be connected. Suppose that $g: X \to V$ is the Stein factorization of $\pi$. Then $\Gamma$ is a very general fiber of $g$. 

2. Notice that by the upper-semicontinuity, $\ka(D|_\Gamma)$ is well-defined for a very general $\Gamma$. By Lemma \ref{le: very general implies general in kod k}, $\ka(D|_\Gamma)$ is a constant on an open set of $V$ if the relative $D$-dimension is non-negative.
\end{remark}

In the following, for a very general fiber $F$, we write $\ka(D|_F)$ to denote $\ka(D|_\Gamma)$ where   $\Gamma$ is a connected component  of $F$.

\begin{lemma}[{c.f. \cite[II Theorem-Definition 3.14]{Nak04}}]\label{le: rel kod and kod of general fiber}
Let $\pi: X\to U$ be a projective morphism between normal varieties and $D$ be an $\Rr$-Weil divisor on $X$. Suppose that the relative $D$-dimension $\ka_{\rm rel}(D; X/U) \geq 0$, then for a general fiber $F$ of $\pi$, we have $\ka(D;X/U) = \ka(D|_F)$.
\end{lemma}
\begin{proof}
First, by the same argument as in Lemma \ref{le: fiber kad positive}, we have $\pi_*\Oo_{X}(\lf mD \rf)=0$ iff $h^0(F, \Oo_{F}(\lf mD \rf|_F))=0$. Hence, if $\pi_*\Oo_{X}(\lf mD \rf)=0$ for any $m$, then $\ka(D|_F) = - \infty$ for very general $F$, which contradicts to $\ka_{\rm rel}(D; X/U) \geq 0$. Hence, $\ka(D;X/U) \neq -\infty$.

Suppose that $\phi=\phi_{U, |mD|}: X \dto \Pp_U(\pi_*\Oo_X(\lfloor mD \rfloor))$ for a sufficiently large and divisible $m$. Passing to a resolution, we can assume that $\phi$ is a morphism. Let $F_\pi$ be an irreducible component of a general fiber of $\pi$ and $G= \phi|_{F_\pi}(F_\pi)$. Then $G$ is an irreducible component of a general fiber of $W_m \to U$. Let $F_\phi$ be a general fiber of $\nu=\phi|_{F_\pi}: F_\pi \to G$ which is also a general fiber of $\phi$. By Lemma \ref{le: general fiber has Kod=0}, we have $\ka(D|_{F_\phi})=0$. By the easy addition (\cite[II Theorem 3.13]{Nak04}), $\ka(D|_{F_\pi}) \leq \ka(D|_{F_\phi}) + \dim G$, we have 
\begin{equation}\label{eq: compare dim}
\ka(D|_{F_\pi}) \leq \dim G = \dim W_m-\dim U = \ka(D; X/U).
\end{equation}

Conversely, for a suifficiently large and divisible $r$, let $\phi_{|rD|_{F_\pi}|}: F_\pi \dto V_r$ be a dominant rational map such that $\dim V_r = \ka(D|_{F_\pi})$. Taking a resolution of $F_\pi$, we can assume that $\phi_{|rD|_{F_\pi}|}$ is a morphism. Let $\Theta$ be a general fiber of $\phi_{|rD|_{F_\pi}|}$, then by Lemma \ref{le: general fiber has Kod=0}, $\ka(D|_\Theta) =0$. Recall that $\nu: F_\pi \to G$ is the morphism defined by restriction $\phi=\phi_{U, |mD|}$ to $F_\pi$. Hence $\ka(D|_\Theta) =0$ implies that $\nu(\Theta)$ is a point. By Lemma \ref{le: contract to point}, there exists $\rho: V_r \dto G$ such that $\nu= \rho \circ \phi_{|rD|_{F_\pi}|}$. In particular, $\dim G \leq \dim V_r$. Hence $$\ka(D|_{F_\pi}) = \dim V_r \geq \dim G = \ka(D; X/U)$$ which is the inverse inequality. 
\end{proof}

\begin{proposition}\label{prop: definition of kod dim coincides}
Let $\pi: X\to U$ be a projective morphism between normal varieties and $D$ be an $\Rr$-Weil divisor on $X$, then $\ka(D; X/U) = \ka_{\rm rel}(D; X/Y)$. 
\end{proposition}
\begin{proof}
Assuming that $\ka_{\rm rel}(D; X/Y)\geq 0$, then by Lemma \ref{le: rel kod and kod of general fiber}, we have $\ka(D; X/U) = \ka_{\rm rel}(D; X/Y)$. Assuming that $\ka(D; X/U) \geq 0$, then $\ka(D|_F) \geq 0$ for general fibers, and thus $\ka_{\rm rel}(D; X/Y)\geq 0$ by the definition of relative $D$-dimensions. Hence, we still have $\ka(D; X/U) = \ka_{\rm rel}(D; X/Y)$. In particular, this implies that $\ka(D; X/U) = -\infty$ iff $\ka_{\rm rel}(D; X/Y) = -\infty$.
\end{proof}

Under the above notation, let $\eta \in U$ be the generic point of $X$, and $\overline{\Oo}_{U,\eta}$ be the algebraic closure of ${\Oo}_{U,\eta}$. Let $\bar\eta = \spec\overline{\Oo}_{U,\eta}$, and $X_{\bar\eta}$ be the geometric fiber of $\pi$. Suppose that $X \to U_{\rm stein} \to U$ is the Stein factorization with $\eta_s$ the generic point of $U_{\rm stein}$. Let $D_{\bar\eta}$ and $D_{\bar\eta_s}$ be the pullbacks of $D$ to $X_{\bar\eta}$ and $X_{\bar\eta_s}$ respectively. We define $\ka(D_{\bar \eta})$ to be $\ka(D_{\bar \eta_s})$. Although $X_{\bar\eta}$ may not be connected, $X_{\bar\eta_s}$ is always integral (see explanations in the second paragraph of Section \ref{sec: relative Kodaira dim})

\begin{proposition}\label{prop: kod general and geometric fiber}
Let $\pi: X\to U$ be a projective morphism  between normal varieties and $D$ be an $\Rr$-Weil divisor on $X$. For a general fiber $F$ of $\pi$, we have $\ka(D|_F)=\ka(D_{\bar\eta})$.
\end{proposition}
\begin{proof}
Let $X \to U_{\rm stein} \to U$ be the Stein factorization. By the definition of $\ka(D_{\bar \eta})$, we have $\ka(D_{\bar \eta})= \ka(D_{\bar \eta_s})$. By the definition of $\ka(D|_F)$ and the discussion above, $\ka(D|_F)=\ka(D|_G)$ where $G$ is a general fiber of $X \to U_{\rm stein}$. Hence replacing $U$ by $U_{\rm stein}$, we can assume that $\pi$ has connected fibers. Taking a resolution of $X$ and shrinking $U$, we can assume that $X$ is smooth and $U$ is affine. 

First, by Proposition \ref{prop: definition of kod dim coincides}, suppose that $\ka(D|_F) = \ka(D; X/U) \geq 0$. Consider $\phi=\phi_{U, |mD|}: X \dto W_m/U$ for a sufficiently large and divisible $m$. Taking a resolution, we can assume that $\phi$ is a morphism, and passing to $\bar \eta$, we have $\phi_{\bar\eta}: X_{\bar\eta} \to W_{m,\bar\eta}$ over $\bar\eta$. Moreover, $\phi_{\bar\eta}$ is defined by a sub-linear system of $|mD_{\bar\eta}|$. In particular, $\ka(D_{\bar\eta}) \geq \dim W_{m,\bar\eta} = \dim W_m-\dim U$. On the other hand, by \eqref{eq: compare dim}, $\ka(D|_{F}) \leq \dim W_m - \dim U$, and thus $\ka(D|_F) \leq \ka(D_{\bar\eta})$.

Conversely, suppose $\ka(D_{\bar\eta}) \neq -\infty$. Let $\phi_{|mD_{\bar\eta}|}: X_{\bar\eta} \dto Z_m$ be the dominant rational map defined by ${|mD_{\bar\eta}|}$ such that $\dim Z_m = \ka(D_{\bar\eta})$. Notice that $X_{\bar\eta}, Z_m$ are finite type over $\overline{K(U)}$, and $\phi_{|mD_{\bar\eta}|}$ is defined by finitely many sections. By shrinking $U$, and throwing coefficients of defining equations of $X_{\bar\eta}, Z_m, \phi_{|mD_{\bar\eta}|}$ in $H^0(U, \Oo_U)$, we can assume (see \cite[\S 3 Lemma 2.6]{Liu02}) that there exist finite covers $\theta: \ti U \to U, \tau: \ti X \to X$, a morphism $\ti \pi: \ti X \to \ti U$ and a rational map $\ti \phi: \ti X \dto \ti Z_m$, such that the following diagram commutes:
\[
 \xymatrix{
  & \ti X  \ar@{-->}[dl]_{\ti\phi} \ar[dd]^{\ti\pi} \ar[r]^\tau & X \ar[dd]^\pi\\
\ti Z_m \ar[dr] & &\\
& \ti U \ar[r]^\theta & U.
}
 \] Replacing $\ti U$ by its normalization and $\ti X, \ti Z_m$ by the corresponding base change, we can assume that $\ti U$ is normal. Moreover, $\phi_{|mD_{\bar\eta}|}$ is obtained by the base change of $\ti \phi$ through $\bar\eta \to \ti U$. Let $\ti D =\tau^*D$, then $\ti \phi = \phi_{L}$ where $L \subset |m\ti D|$ is a sub-linear system. Let $\ti F$ be a general fiber of $\ti\pi$ and $F =\tau (\ti F)$ which is a general fiber of $\pi$. Then $L|_{\ti F}$ induces the rational map 
 \[
 \phi_{L|_{\ti F}}: \ti F \dto \ti\phi(\ti F)
 \] such that $\dim \ti\phi(\ti F) = \dim \ti Z_m - \dim \ti U=\dim Z_m=\ka(D_{\bar \eta})$. Let $n$ be the degree of the finite morphism $\tau_{\ti F}: \ti F \to F$. By $${\tau_{\ti F}}_*{\tau_{\ti F}}^*(D|_F) = n (D|_{\ti F}),$$ we have $|\tau_{{\ti F},*}( \lf \tau^*mD\rf|_{\ti F})| \subset |mnD|_F|$. In particular, $|\tau_{{\ti F},*} L|_{\ti F}| \subset |mnD|_F|$. Hence $$\ka(D|_F) \geq \dim \phi_{|mnD|_F|} (F) \geq \dim \ti\phi(\ti F) = \ka(D_{\bar\eta}).$$
 
 Finally, the above argument implies that $\ka(D_{\bar\eta})=-\infty$ iff $\ka(D_F)=-\infty$, and thus completes the proof.
\end{proof}

Combining Proposition \ref{prop: definition of kod dim coincides} with Proposition \ref{prop: kod general and geometric fiber}, we have the following result.

\begin{theorem}\label{thm: compare various definition of relative Iitaka dim}
Let $\pi: X\to U$ be a projective morphism between normal varieties and $D$ be an $\Rr$-Weil divisor on $X$, then 
\[
\ka(D; X/U) = \ka_{\rm rel}(D; X/Y)=\ka(D_{\bar \eta}).
\]
\end{theorem}

We adopt the following definition of relative Iitaka fibrations (see \cite[Definition 2.1.34]{Laz04I}).

\begin{definition}[Relative Iitaka fibration]\label{def: Iitaka fibration}
Let $\pi: X\to U$ be a projective morphism between varieties. Suppose that $X$ is normal and $D$ is an $\Rr$-Weil divisor on $X$ such that $\kappa(D; X/U) \geq 0$. 
Let $W_m = \phi_{U, |mD|}(X)$. Then any morphism $f: Y \to Z/U$ between smooth varieties satisfying the following property is called a relative Iitaka fibration associated with $D$ over $U$:
\begin{enumerate}
\item there is a birational morphism $h: Y \to X$, and a birational map $g_m: Z \dto W_m$ for each $m$ sufficiently large and divisible, so that the following diagram commutes
\[
 \xymatrix{
 Y \ar[rr]^h \ar[d]_f && X \ar@{-->}[d]^{\phi_{U, |mD|}}\\
 Z \ar@{-->}[rr]^{g_m} \ar[dr] & & W_m \ar[dl]\\
 & U &,
}
 \] 
\item if $D'=h_*^{-1}D+n\Exc(h)$ for $n \gg 1$, then $\dim Z = \kappa(D'; Y/U)+\dim U$, and 
\item $\ka(D'|_{F_{f}})=0$, where $F_{f}$ is a general fiber of $f$.
\end{enumerate}
\end{definition}

\begin{remark}
In \cite[Definition 2.1.34]{Laz04I}, condition (3) is stated for very general fibers instead of general fibers. However, by Lemma \ref{le: very general implies general in kod k}, these two conditions are equivalent because $\ka(D; X/U) \geq 0$.
\end{remark}

\begin{proposition}
Under the notation and assumptions of Definition \ref{def: Iitaka fibration}, an Iitaka fibration exists, and Iitaka fibrations are unique up to birational equivalence. 
\end{proposition}
\begin{proof}
Take $m \in\Zz_{\geq 0}$ sufficiently large and divisible and let $W_m = \phi_{U, |mD|}(X)$. Taking a resolution, we can assume that $\phi_{U, |mD|}(X)$ is a morphism. By Lemma \ref{le: general fiber has Kod=0}, $\ka(D|_F)=0$ for any general fiber $F$ of $\phi_{U, |mD|}(X)$. By Lemma \ref{le: compare images of iitaka fibration}, for a sufficiently large and divisible $k$, there is a birational morphism $\tau_{m,k}: W_m \dto W_{k}/U$ such that $\tau_{m,k}\circ\phi_{U, |mD|}=\phi_{U, |kD|}$. Take a resolution of $Z \dto W_m/U$ and a resolution $Y$ of the main component of $Z \times_{W_m} X$. Let $h: Y \to X/U$ and $f: Y \to Z/U$ be the corresponding morphisms. By $\ka(D|_F)=0$ again, we have $\ka(D'|_{F_f})=0$ where $D'=h_*^{-1}D+n\Exc(h)$ for $n \gg 1$ and $F_f$ is a general fiber of $f$ (for example, see \cite[Chapter II, Lemma 3.11]{Nak04}). By $h_*\Oo_Y(\lf k D' \rf) = \Oo_X(\lf kD \rf)$ for any $k\in \Zz_{\geq 0}$, $$\ka(D';Y/U) = \ka(D; X/U) = \dim W_m -\dim U = \dim Z -\dim U.$$ Thus, we obtain a relative Iitaka fibration by Definition \ref{def: Iitaka fibration}. Such relative Iitaka fibrations are unique up to birational equivalence by definition. 
\end{proof}

\subsection{Iitaka fibrations versus ample models}\label{subsec: Iitaka fibrations versus ample models}

We study ample models between varieties. The following lemma should be well-known.

\begin{lemma}\label{le: compare R-linear systems}
Let $g: X \to Z/U$ be a projective morphism with connected fibers between varieties over $U$. Assume that $Z$ is a $\Qq$-factorial variety and $E \geq 0$ is a divisor on $X$ such that $g_*\Oo_X(\lf k E \rf)=\Oo_Z$ for any $k \in \Zz_{\geq 0}$. Then for any $\Rr$-divisor $D$ on $Z$, there is a natural identification between following $\Rr$-linear systems
\[
|D/U|_{\Rr} \xrightarrow{\sim} |g^*D+E/U|_{\Rr}.
\]
\end{lemma}
\begin{proof}
There is a natural map $B \mapsto g^*B+E$ for any $B \in |D/U|_{\Rr}$.  This map is clearly injective. For surjectivity, take $C \in |g^*D+E/U|_{\Rr}$. Then by definition
\begin{equation}\label{eq: expression of C}
C= g^*D+E+\sum_i r_i ~{\rm div}(s_i)+g^*\phi^*\Theta,
\end{equation} where $r_i \in \Rr \backslash\{0\}$, $s_i \in K(X)\backslash\{0\}$ are rational sections, $\Theta$ is an $\Rr$-Cartier divisor on $U$, and $\phi: Z \to U$ is the morphism. Let $F$ be a very general fiber of $g: X \to Z$. By $g_*\Oo_X(\lf k E \rf)=\Oo_Z$ for any $k \in \Zz_{\geq 0}$, we know that $\ka(E; X/Z)=\ka(E|_{F})=0$ by Theorem \ref{thm: compare various definition of relative Iitaka dim}. Hence $H^0(F, \Oo_F(\lf kE \rf|_{F}))=\Cc$ for each $k \in \Zz_{\geq 0}$. By \eqref{eq: expression of C}, we have
\[
0\leq C|_F = E|_F + \sum_i r_i ~{\rm div}(s_i|_F).
\] If $\sum_i r_i ~{\rm div}(s_i|_F) \neq 0$, then we claim that there exists $r_i' \in \Qq$ close enough to $r_i$ for each $i$ such that $\sum_i r'_i ~{\rm div}(s_i|_F)) \neq 0$ either. Moreover, we can assume that $2E|_F + \sum_i r'_i ~{\rm div}(s_i|_F)) \geq 0$. In fact, suppose that $P$ is a prime divisor such that $P \subset \Supp(\di(s_i|_F))$ for some $i$ and $P \not\subset \Supp E|_F$, $P \not\subset \Supp C|_F$. Let $n_{P, i} = \mult_P \di(s_i|_F)$, we have $\sum_i r_i n_{P,i}=0$. There are finitely many such divisors $P$, and hence finitely many such rational equations. Therefore, as long as $\{r_i' \mid i\}$ is a solution of these equations and $|r_i'-r_i| \ll 1$, they satisfy the claimed property.

Let $m \in \Zz_{>0}$ such that $mr_i' \in \Zz$ for each $i$, and choose $k \in \Zz_{>0}$ sufficiently large, we have
\[
\lf k E|_F \rf + \sum mr_i' ~{\rm div}(s_i|_F) \geq 0. 
\] Thus $H^0(F, \Oo_F(\lf kE \rf|_{F}))\supsetneq \Cc$ which is a contradiction. 

Hence $\sum_i r_i ~{\rm div}(s_i|_F) = 0$, and thus $\sum_i r_i{\rm div}(s_i)$ can only be vertical over $Z$. Let $E^h$ be the horizontal part of $E$ over $Z$, we have $C \geq E^h$. Because $E^v=E-E^h$ satisfies the same property of $E$ in the proposition, replacing $E$ by $E^v$ and $C$ by $C-E^h$, we can assume that $E$ is vertical over $Z$. If $\codim g(E)=1$, then $g_*\Oo_X(\lf kE \rf) \supsetneq \Oo_Z$ for $k \gg 1$. Hence $\codim g(E) \geq 2$. 

For any prime divisor $P$ on $Z$, define 
\begin{equation}\label{eq: def of t}
t_P(C) \coloneqq \max\{t \in \Rr_{\geq 0} \mid C-tg^*P \geq 0 \text{~over the generic point of~} P\}.
\end{equation} Let $B = \sum_P t_P(C)P$ on $Z$. Then $C - g^*B =\Gamma$ with $\codim g(\Gamma) \geq 2$. By $\Gamma = C - g^*B \sim_{Z, \Rr} E$ and $\codim g(E) \geq 2$, we have $\Gamma = E$. Hence $C = g^*B +E$.
\end{proof}

Lemma \ref{le: compare R-linear systems} implies that the ample model does not change under resolutions. Precisely, we have the following lemma.

\begin{lemma}\label{le: compare ample models}
Let $g: X \to Z/U$ be a projective morphism with connected fibers between varieties over $U$. Assume that $Z$ is a $\Qq$-factorial variety and $E \geq 0$ is an $\Rr$-Cartier divisor on $X$ such that $g_*\Oo_X(\lf k E \rf)=\Oo_Z$ for any $k \in \Zz_{\geq 0}$. Let $D$ be an $\Rr$-Cartier divisor on $Z$. 
\begin{enumerate}
\item If $h: Z \dto V$ is the ample model$/U$ for $D$, then $h \circ g: X \dto V$ is the ample model$/U$ for $g^*D+E$. 
\item Suppose that $\ka(D; Z/U) \geq 0$. If $\tau: X \dto V$ is the ample model$/U$ for $g^*D+E$, then there is a rational map $h: Z \dto V/U$ such that $h$ is the ample model$/U$ for $D$ and $h \circ g = \tau$.
\item If $g$ is a birational morphism, then the $\Qq$-factorial assumption on $Z$ is not necessary: $Z \dto V$ is the ample model$/U$ for $D$ iff $X \dto V$ is an ample model$/U$ for $g^*D+E$.
\end{enumerate}
\end{lemma}
\begin{proof}
We show (3) first. By $g_*\Oo_X(\lf k E \rf)=\Oo_Z$ for any $k \in \Zz_{\geq 0}$, we have $\codim g(E) \geq 2$. Suppose that $Z \dto V$ is the ample model for $D$.  Let $X \xleftarrow{p} W \xrightarrow{q} V$ be a resolution of $X \dto V$. Let $H$ be the ample$/U$ divisor on $V$ such that $p^*g^*D \sim_{U, \Rr} q^*H+F$ with $F \geq 0$ such that for any $B \in |p^*g^*D/U|_\Rr$, $B \geq F$. If $B' \in  |p^*g^*D+p^*E/U|_\Rr$, then $g_*p_*B' \in |D/U|_\Rr$ and thus $p^*g^*(g_*p_*B') \in |p^*g^*D/U|_\Rr$. In particular, $p^*g^*(g_*p_*B') \geq F$. On the other hand, by $B' - p^*g^*(g_*p_*B') \sim_{U, \Rr} p^*E$ and $B' - p^*g^*(g_*p_*B')$ is $(g \circ p)$-exceptional, we have $B' - p^*g^*(g_*p_*B') = p^*E$. Thus $B' \geq F+p^*E$. Notice that $p^*g^*D+p^*E \sim_{U, \Rr} q^*H+F+p^*E$, hence $X \dto V$ is the ample model for $g^*D+E$.

Conversely, suppose that $X \dto V$ is the ample model for $g^*D+E$. We show that it is also the ample model for $g^*D$. Indeed, if $p: W \to X$ is a resolution, then $|p^*g^*D+p^*E/U|_\Rr = |p^*g^*D/U|_\Rr+p^*E$ by the same argument as above. Next, by $g^*|D/U|_\Rr = |g^*D/U|_\Rr$, $Z \dto V$ is the ample model for $D$.

For (1), let $Z \xleftarrow{p} W \xrightarrow{q} V$ be a resolution of $h$, and  let $X \xleftarrow{\theta} X' \rightarrow V$ be a resolution of $h \circ g$ such that $X' \to W$ is a morphism. By (3), $X \dto V'$ is the ample model for $g^*D+E$ iff $X' \to X \dto V'$ is the ample model for $\theta^*(g^*D+E)$. Replacing $X$ by $X'$, $E$ by $\theta^*E$ and $g^*D+E$ by $\theta^*g^*D+\theta^*E$, we can assume that $\tau: X \to V$ is a morphism. Let $\mu: X \to W$ be the morphism such that $\tau = q \circ \mu$. Because $Z \dto V$ is the ample model for $D$, there exists an ample divisor $A$ on $V$ and $F \geq 0$ on $W$ such that $q^*A+F \in |p^*D/U|_\Rr$. Moreover, if $B' \in |p^*D/U|_\Rr$ then $B' \geq F$. Thus $\tau^* A+(\mu^*F+E) \in |g^*D+E/U|_\Rr$. If $B \in  |g^*D+E/U|_\Rr$, then by Lemma \ref{le: compare R-linear systems}, there exists $D' \in  |D/U|_\Rr$ such that $B =g^*D'+E$. Hence $B \geq \mu^*F+E$. $\tau$ is a contraction morphism because $\mu$ and $q$ are contraction morphisms. This shows that $\tau: X \to V$ is the ample model for $g^*D+E$ over $U$ by definition.

For (2), let $\phi: Z \dto W_m/U$ be an Iitaka fibration of $D$ for a sufficiently large and divisible $m$. Then by $g_*\Oo_X(\lf k E \rf)=\Oo_Z$ for any $k \in \Zz_{\geq 0}$, we claim that $X \to Z \dto W_m/U$ is birational to an Iitaka fibration of $g^*D+E$. It is enough to show $g_*\Oo_X(\lf k(g^*D+E) \rf) = \Oo_Z(\lf kD \rf)$ for any $k \in \Zz_{\geq 0}$. The ``$\supset$'' is by definition. We show the converse inclusion. Notice that over $V= Z \backslash \Supp D$, $g_*\Oo_X(\lf k(g^*D+E) \rf)|_V = \Oo_Z(\lf kD \rf)|_V= \Oo_V$, hence for any open set $U \subset Z$, we can identify a section $\alpha$ in $g_*\Oo_X(\lf k(g^*D+E) \rf)|_U$ with a rational section $ \beta \in K(Z)$ through $\alpha=\beta \circ g$. Suppose that there exists $\alpha \in K(X)$ such that $\di(\alpha)+k(g^*D+E) \geq 0$ over an open set $U \subset Z$. By Lemma \ref{le: compare R-linear systems}, there exists $\beta \in K(Z)$ such that $\alpha=\beta \circ g$ and $\di(\beta)+kD \geq 0$. In fact, 
\[
g^*(\di(\beta)/k+D)+E =\di(\alpha)/k+(g^*D+E) \in |g^*D+E/U|_\Rr,
\] and thus $\di(\beta)/k+D \in |D/U|_\Rr$. In conclusion, $g_*\Oo_X(\lf k(g^*D+E) \rf) \subset \Oo_Z(\lf kD \rf)$, and this shows the claim.

By Proposition \ref{prop: ample model bir to Iitaka} below, there is a birational map $W_m \dto V/U$. This induces a map $h: Z \dto V/U$ such that $h \circ g = \tau$. Let $\psi: Z \to U$ be the morphism. As before, $Z \xleftarrow{p} W \xrightarrow{q} V$ is a resolution of $h$, and by taking a resolution, we can assume that there are morphisms $\mu: X \to W$ and $\tau = q \circ \mu$. Let $A >0$ be the ample divisor on $V$ by the definition of the ample model for $g^*D+E$. Then $g^*D+E = \mu^*p^*D+E$ and $\mu^*q^*A+F \in |\mu^*p^*D+E/U|_\Rr$ for $F \geq 0$ such that for any $B \in  |\mu^*p^*D+E/U|_\Rr$, $B \geq F$. We claim $F \geq E$. In fact, by Lemma \ref{le: compare R-linear systems}, $$|p^*D/U|_\Rr \xrightarrow{\sim} |\mu^*p^*D/U|_\Rr\xrightarrow{\sim} |\mu^*p^*D+E/U|_\Rr.$$ Thus for any $\Gamma \in |\mu^*p^*D+E/U|_\Rr$, $\Gamma \geq E$. If $F \not\geq E$, then by $A$ ample, we can find $0<A' \sim_\Rr A$ such that $\Supp\mu^*p^*A'$ does not contain any component of $E$. Thus $\mu^*p^*A'+F \not\geq E$, a contradiction. By Lemma \ref{le: compare R-linear systems}, there exists $B \in |p^*D/U|_{\Rr}$ such that $\mu^*q^*A+F = \mu^*B+E$. Hence $\Theta \coloneqq B-q^*A$ satisfies that $\mu^*\Theta= F-E \geq 0$, and thus $\Theta \geq 0$. Notice that $q: W \to V$ is a contraction. Finally, for any $H \in |p^*D/U|_\Rr$, $\mu^*H+E \in |\mu^*p^*D+E/U|_\Rr$, and thus $\mu^*H+E \geq F$ because $\tau: X \dto V$ is the ample model for $\mu^*p^*D+E$ over $U$. Thus $\mu^*H \geq F-E=\mu^*\Theta$ and hence $H \geq \Theta$. Notice that $\mu^*q^*A+F-E \in |\mu^*p^*D/U|_\Rr$, hence $q^*A+\Theta \in |p^*D/U|_\Rr$. This shows that $h: Z \dto V/U$ is the ample model for $D$.
\end{proof}

\begin{proposition}\label{prop: ample model bir to Iitaka} 
Let $\pi: X\to U$ be a projective morphism between varieties. Assume that $X$ is normal and $D$ is an $\Rr$-Cartier divisor on $X$ such that $\ka(D; X/U) \geq 0$. If $f: X \dto Z/U$ is the ample model for $D$, then $f$ is birational to an Iitaka fibration associated with $D$.
\end{proposition}
\begin{proof}
Replacing $X$ by a resolution of $f$ and $D$ by its pullback, we can assume that $f$ is a morphism. Moreover, $f$ is a contraction by the definition of the ample model. 

First, we claim that for a general fiber $F$ of $f$, $\ka(D|_F) =0$. By the definition of the ample model, $D \sim_{U,\Rr} f^*H+E$ where $H$ is an ample$/U$ divisor on $Z$  and $E \geq 0$ such that for any $B \in |D/U|_\Rr$, $B \geq E$. By $\ka(D;X/U) \geq 0$, we have $\ka(D|_F) \geq 0$. If $\ka(D|_F)>0$, then $\ka(E|_F) = \ka(D|_F)>0$ by Proposition \ref{prop: Choi}. By Lemma \ref{le: fiber kad positive}, there is a sufficiently large and divisible $m$ such that the a rational map $\phi=\phi_{Z, |mE|}: X \dto Z'/Z$ satisfying $\dim Z'>\dim Z$. Possibly taking a resolution, we can assume that $\phi$ is a morphism.  Hence, there is an ample divisor $A$ over $Z$ on $Z'$ such that $E \sim_{Z,\Rr} \phi^*A+E'$ for some $E>E' \geq 0$. In other words, there exists $L$ on $Z$ such that $E \sim_{\Rr} \phi^*A+E'+f^*L$. By the choice of $H$, take $\ell \gg 1$ such that $A+\frac 1 2\ell h^*H$ is ample over $U$ and $\frac 1 2 \ell H+L$ is ample over $U$. Because 
\[
\begin{split}
(\ell +1)D &\sim_{U, \Rr} f^*H+(\ell f^*H+E)+\ell E\\
&\sim_{U, \Rr} f^*H+(\ell f^*H+\phi^*A+E'+f^*L)+\ell E\\
&\sim_{U, \Rr} f^*H+(\frac 1 2 \ell f^*H+\phi^*A)+(\frac 1 2 \ell f^*H+f^*L)+E'+\ell E, 
\end{split}
\] where $f^*H+(\frac 1 2 \ell f^*H+\phi^*A)+(\frac 1 2 \ell f^*H+f^*L)$ is semi-ample over $U$. Hence there exists $0<B' \sim_{U, \Rr} f^*H+(\frac 1 2 \ell f^*H+\phi^*A)+(\frac 1 2 \ell f^*H+f^*L)$ such that $\Supp B'$ does not have any component of $\Supp E$. Because $E'<E$, $\frac{1}{\ell+1}(E'+\ell E)<E$. Hence $B = \frac{1}{\ell+1}B'+\frac{1}{\ell+1}(E'+\ell E) \in |D/U|_{\Rr}$. But $B \not\geq E$, a contradiction. This shows Definition \ref{def: Iitaka fibration} (3)

Next, we show that for any $m$ sufficiently large and divisible such that $\phi_{U, |mD|}: X \dto W_m$, there is a birational map $g_m: Z \dto W_m$ satisfying $g_m \circ f = \phi_{U, |mD|}$. The existence of $g_m$ can be shown by a similar argument as that for Lemma \ref{le: compare images of iitaka fibration}, so we just sketch the argument below. By $\ka(D|_F)=0$ and Lemma \ref{le: contract to point}, there is a map $g_m : Z \dto W_m$. We show that it is a birational map. By the definition of the ample model, there is an ample$/U$ divisor $H$ such that $D \sim_{U, \Rr} f^*H+E$ with $E \geq 0$ and for any $B \in |D/U|_\Rr$, $B \geq E$. Let $G$ be a general fiber of $\pi$, then $D|_G \sim_\Rr (f^*H)|_G+E|_G$. By Proposition \ref{prop: Choi}, $\ka(D|_G)=\ka((f^*H)|_G+E|_G)$. Applying the easy addition (\cite[II Theorem 3.13]{Nak04}) to $f|_G: G \to f(G)$, we have
\[
\ka((f^*H)|_G+E|_G) \leq \ka(E|_F)+\dim f(G)
\] where $F$ is a general fiber of $f|_G$. By the argument in the first part, $\ka(E|_F)=0$, thus $\ka((f^*H)|_G+E|_G) \leq \dim f(G)$. But 
\[
\ka((f^*H)|_G+E|_G) \geq \ka((f^*H)|_G) = \ka(f|_G^*(H|_{f(G)})) = \dim f(G).
\] Thus $\ka((f^*H)|_G+E|_G) = \dim f(G) = \dim Z -\dim U$. By Theorem \ref{thm: compare various definition of relative Iitaka dim}, $\ka(D|_G) = \ka(D; X/U)=\dim W_m-\dim U=\dim Z-\dim U$. Hence $\dim W_m = \dim Z$. Because $f$ and $\phi_{U, |mD|}$ are contractions, $g_m$ is a birational map. This shows Definition \ref{def: Iitaka fibration} (1)

Finally, under the notation of Definition \ref{def: Iitaka fibration} (2), suppose that $h: Y \to X$ is a birational morphism and $D' = h_*^{-1}D+n \Exc(h)$ for $n \gg 1$. Then $\ka(D'; Y/U) = \ka(D; X/U)$ (see the third paragraph of Section \ref{sec: relative Kodaira dim}), and by the above argument, we have $\ka(D; X/U)=\dim Z -\dim U$. This shows Definition \ref{def: Iitaka fibration} (2).
\end{proof}

\begin{remark}
Proposition \ref{prop: ample model bir to Iitaka} is false without assuming $\ka(D; X/U) \geq 0$. Remark \ref{rmk: choi} can be used to construct counterexamples. 
\end{remark}

\section{Finiteness of log canonical models and its corollary}\label{sec: finite ample models}

\begin{proposition}\label{prop: same kodaira dim in the interior of a polytope}
Let $\pi: X\to U$ be a projective morphism between normal varieties. Suppose that $Q \subset \WDiv(X)_\Rr$ is a closed convex set (in Euclidean topology) such that $\ka(D; X/U) \geq 0$ for any $D \in Q^\circ$, where $Q^\circ$ is the relative interior of $Q$. Then $\ka(D; X/U)$ is a constant for any $D \in Q^\circ$.
\end{proposition}
\begin{proof}
Let $\De \in Q^\circ$ such that $\ka(\De; X/U)$ achieves the maximal value among all $\ka(D; X/U), D \in Q^\circ$. By Proposition \ref{prop: definition of kod dim coincides}, $\ka(D; X/U) = \ka(D|_F)$ where $F$ is a general fiber of $\pi$. Because $Q$ is convex, $Q$ can be approximated by inscribed polytopes. Thus for any $D \in Q^\circ$, there exists $B \in Q^\circ$ and $r\in (0,1)$ such that $D = r\De+(1-r)B$. As $\ka(B|_F)=\ka(B; X/U) \geq 0$,
\[
\ka(D|_F) \geq \max\{\ka(\De|_F), \ka(B|_F)\} \geq \ka(\De|_F)
\] by definition. Thus $\ka(D; X/U)=\ka(D|_F) = \ka(\De|_F)=\ka(\De; X/U)$.
\end{proof}

\begin{proposition}\label{prop: uniform Iitaka fibration}
Let $\pi: X\to U$ be a projective morphism between normal varieties and $Q \subset \WDiv(X)_{\Rr}$ be a closed convex set. Suppose that $\ka(D/U) \geq 0$ for any $D \in Q^\circ$.  Then there is a rational map $f: Y \dto Z/U$ which is an Iitaka fibration associated with any $D \in Q^\circ$.
\end{proposition}
\begin{proof}
First, we show that when $\ka(D; X/U) = \ka(B; X/U) \geq 0$ for $D \leq B$, then $D$ and $B$ share a same Iitaka fibration (up to birational equivalence). Let $f: Y \to Z/U$ be an Iitaka fibration associated with $D$. By taking a resolution, we can assume that $X=Y$. For a general fiber $F$ of $f$, we have $\ka(B|_F)=0$. In fact, if $\ka(B|_F)>\ka(D|_F)=0$, then just as the proof in Lemma \ref{le: general fiber has Kod=0}, we have $\ka(B; X/U) > \ka(D; X/U)$, a contradiction. For a sufficiently large and divisible $m$, let $\phi_D = \phi_{U, |mD|}: X \dto W_m/U$, and $\phi_B = \phi_{U, |mB|}: X \dto V_m/U$ be the corresponding rational maps. We claim that there exists a birational map $\theta: V_m \dto W_m/U$ such that $\theta \circ \phi_B = \phi_D$. Without loss of generality, we can assume that $ \phi_D, \phi_B$ are morphisms. Let $G$ be general fibers of $\phi_B$. Then by $\ka(B|_{G})=0$ and $D \leq B$, we see that $G$ is contracted to a point by $\phi_D$. Thus by Lemma \ref{le: contract to point}, there exists $\theta: V_m \dto W_m$ such that $\theta \circ \phi_B = \phi_D$. Because $\dim V_m = \dim W_m$ and $\theta$ has connected fibers over general points (as $\phi_D, \phi_B$ have generic connected fibers), $\theta$ is a birational map. By Definition \ref{def: Iitaka fibration}, $f: Y \to Z/U$ is also an Iitaka fibration associated with $B$ over $U$.

By Proposition \ref{prop: same kodaira dim in the interior of a polytope}, $\ka(D; X/U)$ is a constant for any $D \in Q^\circ$. Fix a divisor $\De \in Q^\circ$ and let $D \in Q^\circ$ be an arbitrary divisor. As a convex set can be approximated by inscribed polytopes, we can assume that $\De, D \subset P \subset Q^\circ$, where $P$ is a polytope. Then there is a sequence of divisors in $P$, $\De=\De_0, \De_1, \ldots, \De_k=D$ such that either $\De_i \geq \De_{i+1}$ or $\De_i \leq \De_{i+1}$. The above shows that if $f: Y \to Z/U$ is an Iitaka fibration associated with $\De$, then it is also an Iitaka fibration associated with $\De_i$ for each $i$. In particular, $f: Y \to Z/U$ is an Iitaka fibration associated with $D$. The claim follows as $D$ is chosen arbitrarily. 
\end{proof}

Under the notation and assumptions of Theorem \ref{thm: main}, suppose that $g: Y \to Z/U$ is an Iitaka fibration associated with $K_X+\Delta$ for each $\Delta \in P$ (it exists by Proposition \ref{prop: uniform Iitaka fibration}). There is $\ep>0$ such that $(X, \De)$ is $\ep$-lc for any $\De \in P$. After taking a higher model, we can assume that $h: Y \to X$ is a log resolution of $(X, \Supp \De)$ for any $\De \in P$. Define 
\begin{equation}\label{eq: resolution}
K_Y+\Delta_Y = h^{*}(K_X+\De)+(1-r)\Supp (\Exc(h)),
\end{equation} where $0<r<\ep$ is a rational number. Then $(Y, \De_Y)$ is still klt, and 
\begin{equation}\label{eq: P_Y}
P_Y \coloneqq \{\De_Y \mid \De \in P\}
\end{equation} is a rational polytope.

Suppose that $\De_Y$ is a $\Qq$-divisor. Applying the canonical bundle formula \eqref{eq: canonical bundle relation} to $(Y, \Delta_Y) \to Z$, we have 
\begin{equation}\label{eq: canonical bundle}
K_Y+\De_Y \sim_\Rr g^*(K_Z+D^{\De_Y}_{Z}+M^{\De_Y}_Z)+B^{\De_Y},
\end{equation} where $D^{\De_Y}_Z$ is the divisorial part (see \eqref{eq: divisorial part}) and $M^{\De_Y}_Z$ is the moduli part (see \eqref{eq: moduli part}).  

Suppose that $\{\De_i \mid 1 \leq i \leq k\}$ is the set of vertices of $P$. In particular, each $\De_i$ is a $\Qq$-divisor. Let
\begin{equation}\label{eq: vertex}
b_i(K_Y+\De_i) \sim f^*(b_iL_i) + B^{\De_i}
\end{equation} be as in \eqref{eq: canonical bundle relation} with $b_iL_i$ a Cartier divisor.

Let $r_i \in \Qq_{\geq 0}$ such that $\sum_i r_i =1$. Then for the divisor $\De_Y = \sum r_i \De_i$ and a sufficiently divisible $m \in \Zz_{\geq 0}$, we have
\begin{equation}\label{eq: canonical bundle for De 1}
m(K_Y+\De_Y) \sim f^*(m \sum r_i L_i) + m\sum r_i B^{\De_i}.
\end{equation} 

\begin{proposition}\label{prop: compare canonical bundle formula for Q divs}
Under the above notation and assumptions, if 
\begin{equation}\label{eq: canonical bundle for De 2}
b (K_Y +\De_Y )\sim f^*(b L)+bB^{\De_Y}
\end{equation} as in \eqref{eq: canonical bundle relation}, then 
\[
\sum r_i L_i \sim_\Qq L \text{~and~} \sum r_i B^{\De_i}=B^{\De_Y}.
\] In particular, $f_*\Oo_Y(k (\Supp \sum_i B_+^{\De_i})) = \Oo_Z$ for any $k \in \Zz_{\geq 0}$.
\end{proposition}
\begin{proof}
First, we show $f_*\Oo_{Y}(\lf k  \sum r_i B_+^{\De_i} \rf)=\Oo_Z$ for any $k \in\Zz_{\geq 0}$. By $B_+^{\De_i} \geq 0$, it is enough to show the claim for sufficiently large $k$. For a very general fiber $F$ of $f$, by Proposition \ref{prop: uniform Iitaka fibration} and Lemma \ref{le: general fiber has Kod=0}, we have  $\ka((K_Y+\De_Y)|_F)=0$. Thus 
\[
\ka(\sum r_i B_+^{\De_i}|_F)=\ka((K_Y+\De_Y)|_F) = 0.
\] If $f_*\Oo_{Y}(\lf k  \sum r_i B_+^{\De_i} \rf) \supsetneq\Oo_Z$, then let $\phi=\phi_{Z, |k (\sum r_i B_+^{\De_i})|}: Y \dto W_k/Z$ for a sufficiently large and divisible $k$. Then by the argument for Lemma \ref{le: general fiber has Kod=0}, $W_k \to Z$ is birational. Hence, there exists an open subset $U \subset Z$ such that $H^0(U, f_*\Oo_{Y}(\lf k  \sum r_i B_+^{\De_i} \rf)) = H^0(U, \Oo_Z)$. Therefore, for any open subset $V\subset Z$, $H^0(f^{-1}(V), \Oo_{Y}(\lf k  \sum r_i B_+^{\De_i} \rf))$ can be identified with elements in $K(Z)$. When $f_*\Oo_{Y}(\lf k  \sum r_i B_+^{\De_i} \rf) \supsetneq\Oo_Z$, there is an open subset $V \subset Z$ and a non-regular section (on $V$) $s\in K(Z)$ such that
\[
s\circ f \in H^0(f^{-1}(V), \Oo_{Y}(\lf k  \sum r_i B_+^{\De_i} \rf)).
\] In particular, for $\di(s\circ f)=\di(s\circ f)_{0}-\di(s\circ f)_{\infty}$,  we have $\di(s\circ f)_{\infty}|_{f^{-1}(V)} \neq 0$. Let $\Theta = k  \sum r_i B_+^{\De_i}$. There is no component $P$ of $\Supp \Theta$ such that $\codim f(P)=1$. Otherwise, suppose that $P$ is a component of $\Supp B_+^{\De_j}$, then $f_*\Oo_{Y}(\lf l B_+^{\De_j} \rf) \supseteq \Oo_Z(P) \supsetneq \Oo_Z$ for $l \gg 1$, a contradiction. Hence $\codim f(\Theta^v) \geq 2$ and ${\rm div}(s\circ f)_{\infty}$ is a horizontal divisor. Thus $(s\circ f)|_F$ is not a constant either. But this contradicts to  $\ka(\Theta|_F)=0$. Hence $f_*\Oo_{Y}(\lf k  \sum r_i B_+^{\De_i} \rf)=\Oo_Z$ for any $k \in \Zz_{\geq 0}$. 

Then, pushing forward \eqref{eq: canonical bundle for De 1} and \eqref{eq: canonical bundle for De 2} over $Z\backslash f(\Supp\sum B^{\De_i}_-)$, and combining with $\codim f(\sum r_i B^{\De_i}_-) \geq 2$, we have $\sum r_i L_i \sim_\Qq L$. By Lemma  \ref{le: redundant part equal}, $\sum r_i B^{\De_i}=B^{\De_Y}$.

Finally, if we take $r_i>0$ for each $i$, then $\Supp \sum_i r_i B_+^{\De_i} = \Supp \sum_i  B_+^{\De_i}$. Hence $f_*\Oo_Y(k(\Supp \sum_i  B_+^{\De_i})) = \Oo_Z$ for $k \in \Zz_{\geq 0}$ follows from $$f_*\Oo_Y(\lf k( \sum_i  r_i B_+^{\De_i}) \rf) = \Oo_Z$$ for $k \in \Zz_{\geq 0}$.
\end{proof}

Now, we are ready to prove Theorem \ref{thm: main} and Corollary \ref{cor: canonical model for R-div}.

\begin{proof}[Proof of Theorem \ref{thm: main}]
By Proposition \ref{prop: uniform Iitaka fibration}, there is a projective morphism $f: Y \to Z/U$ between smooth varieties which is the Iitaka fibration associated with $K_X+\De$ for any $\De \in P$. Possibly after taking a higher model of $Y$, we define $\De_Y$ as in \eqref{eq: resolution}. By Lemma \ref{le: compare ample models} (3), the log canonical model for $(Y, \De_Y)$ is also the log canonical model for $(X, \De)$. Replacing $(Y, \De_Y)$ by $(X, \De)$, we can assume that $f: X \to Z$ is a morphism. We use the construction in \cite[\S 4.4]{FM00}:
\begin{enumerate}
\item Let $\Sigma \subset Z$ be an effective divisor such that $f$ is smooth and $\Supp \De^h$ is relatively normal crossing over $Z \backslash \Sigma$,
\item $f(\Supp \De^v) \subset \Sigma$, and
\item  $f$ is flat over $Z \backslash \Sigma$.
\end{enumerate}

Let $\nu: Z' \to Z$ be a birational morphism from a nonsingular variety such that 

\begin{enumerate}[(i)]
\item $\Sigma' = \nu^{-1}(\Sigma)$ is a simple normal crossing divisor,
\item $\nu$ induces an isomorphism $Z' \backslash \Sigma' \simeq Z\backslash \Sigma$, and
\item The irreducible component $X_1$ of $X \times_Z Z'$ dominating $Z'$ is flat over $Z'$.
\end{enumerate}

Let $X'$ be the normalization of $X_1$, and $\tau: X' \to X$, $f': X' \to Z'$ the induced morphisms. Let $g': Y' \to X'$ be a log resolution of $(X, \Supp \De)$ for each $\De \in P$, and set 
\begin{equation}\label{eq: on Y'}
K_{Y'}+\Gamma' = (\tau \circ g')^*(K_X+\De).
\end{equation} By $(X, \De)$ klt for any $\De\in P$, there exists $\ep\in\Qq_{>0}$, such that for $\Upxi \coloneqq f^{-1}_* \De + (1-\ep) \Exc(\tau \circ g')$, $(Y, \Upxi)$ is klt with $\Upxi - \Gamma \geq 0$ exceptional over $X$. Notice that $\Upxi$ is chosen from a rational polytope $Q$. Replacing $Z'$ by $Z$, we can assume that $Z$ is smooth and $X_1 \to Z$ is flat.

Suppose that $\{\Upxi_i \mid 1 \leq i \leq k\}$ is the set of vertices of $Q$. Let $\mu: Y' \to Z$, then by the canonical bundle formula (see \eqref{eq: canonical bundle}),
\begin{equation}\label{eq: for vertices}
K_{Y'}+\Upxi_i \sim_\Rr \mu^*(K_{Z} + D^{\Upxi_i}_{Z}+M^{\Upxi_i}_{Z})+B^{\Upxi_i}. 
\end{equation} Moreover, $M^{\Upxi_i}_{Z}$ is nef$/U$ (see  \cite[Theorem 4.5 (iv)]{FM00} and the discussion in the third paragraph of Subsection \ref{subsec: canonical bundle formula} ) and $(Z, D^{\Upxi_i}_{Z})$ is still klt (\cite[Theorem 3.1]{Amb04}). By $\codim \mu(B_-^{\Upxi_i}) \geq 2$ and (iii) above, $B_-^{\Upxi_i}$ is exceptional over $X'$. Because  $K_{Z} + D^{\Upxi_i}_{Z}+M^{\Upxi_i}_{Z}$ is big$/U$, there is a klt pair $(Z, \Theta_i)$ such that 
\begin{equation}\label{eq: K_Z+Theta = L}
K_{Z}+\Theta_i \sim_{U, \Qq} K_{Z} + D^{\Upxi_i}_{Z}+M^{\Upxi_i}_{Z}.
\end{equation} This $(Z, \Theta_i)$ can be also obtained by $M^{\Upxi_i}_{Z}$ abundant as in \cite[Theorem 0.2]{Amb05}.

Because $K_Z+\Theta_i$ is big$/U$ and klt, there exists a sufficiently general ample$/U$ $\Qq$-Cartier divisor $A>0$ such that there exists $E_i \geq 0$ satisfying $A+E_i \in |K_{Z}+\Theta_i/U|_\Qq$. Take $\delta \in \Qq_{>0}$ such that $(Z, \Theta_i+\delta (A+E_i))$ is still klt, then we have
\[
(1+\delta)(K_{Z}+\Theta_i) \sim_{U, \Qq } K_{Z}+\Theta_i+\delta (A+E_i).
\] For any $\Upxi = \sum r_i \Upxi_i$ with $\sum r_i  =1, r_i \in \Rr_{\geq 0}$. Set $\Theta \coloneqq \sum r_i \Theta_i$, then
\[
K_{Z}+\Theta = \sum_i r_i (K_{Z} + \Theta_i),
\] is klt. By \eqref{eq: for vertices} and \eqref{eq: K_Z+Theta = L},
\[
\begin{split}
K_{Y'} + \Upxi = \sum_i r_i(K_{Y'}+\Upxi_i) &\sim_{U, \Rr} \sum_i r_i  \mu^*(K_Z+\Theta_i)+\sum_i r_i B^{\Upxi_i}\\
&=\mu^*(K_Z+\Theta) + \sum_i r_i B^{\Upxi_i}.
\end{split}
\]

By Proposition \ref{prop: compare canonical bundle formula for Q divs}, $\mu_*\Oo_{Y'}(\lf k \sum_i r_i B_+^{\Upxi_i} \rf) = \Oo_Z$ for any $k \in \Zz_{\geq 0}$. Thus, the log canonical model for $(Z, \Theta)$ is the ample model for $K_{Y'} + \Upxi+\sum_i r_i B_-^{\Upxi_i}$ by Lemma \ref{le: compare ample models} (1). By \eqref{eq: on Y'}, $\Upxi - \Gamma \geq 0$ and $\sum_i r_i B_-^{\Upxi_i}$ is exceptional over $X$, using Lemma \ref{le: compare ample models} (3), we see that the ample model for $K_{Y'} + \Upxi+\sum_i r_i B_-^{\Upxi_i}$ is the log canonical model for $(X, \De)$. By
\[
(1+\delta )(K_{Z}+\Theta) \sim_{U, \Rr} K_{Z} + (\Theta+\delta \sum r_i E_i) + \delta A,
\] the log canonical model for $(Z, (\Theta+\delta \sum r_i E_i) + \delta A)$ is the log canonical model for $(Z, \Theta)$, and hence the log canonical model for $(X, \De)$ by the above discussion.

Let $V_A =\{\sum r_i \Theta_i+\delta \sum r_i E_i + \delta A \mid r_i \geq 0, \sum_{i=1}^k r_i =1\}$. By \cite[Corollary 1.1.5]{BCHM10}, there are finitely many rational maps $\psi_j: Z \dto Z_j/U, 1 \leq j \leq p$, such that 
\[
V_A= \cup_{j=1}^p \Aa_{\psi_j, \pi}(V_A).
\] Moreover, each $\bar\Aa_{\psi_j, \pi}(V_A)$ is a finite union of rational polytopes. Notice that $\De= \sum r_i \De_i \in P$ corresponds to $\vartheta (\De) \coloneqq \sum r_i \Theta_i+\delta \sum r_i E_i + \delta A \in V_A$. Thus $\vartheta$ is a linear map between divisors. Set $\phi_j = \psi_j \circ f$, then 
\[
\Aa_{\phi_j, \pi}(P) = \{\De \in P \mid \vartheta(\De) \in  \Aa_{\psi_j, \pi}(V_A)\}.
\] Therefore, $P = \cup_{j=1}^p \Aa_{\phi_j, \pi}(P)$ and $\bar\Aa_{\phi_j, \pi}(P)$ is a finite union of rational polytopes.
\end{proof}

\begin{remark}\label{rmk: relatively compact}
We have to work with the relatively compact subset $P$ instead of $Q$ because for $\De$ on the boundary of $Q$, $\kappa(K_X+\De; X/U)$ may become smaller. Therefore, for the fixed Iitaka fibration $X \to Z$, $K_Z+D^\De_Z+M_Z^\De$ is no longer big$/U$, and thus the argument breaks. 
\end{remark}

\begin{remark}\label{rmk: constant Iitaka is enough}
Theorem \ref{thm: main} still holds true if we assume that $Q$ is a rational polytope such that $\ka(K_X+\De; X/U) \geq 0$ are constant for all $\De \in Q$. Under this assumption, there is no need to work with the relatively compact subset $P$.
\end{remark}

\begin{proof}[Proof of Corollary \ref{cor: canonical model for R-div}]
Taking a resolution of $X$, we may assume that $X$ is smooth by Lemma \ref{le: compare ample models} (3). It is well-known that
\[
R \coloneqq \{D \geq 0 \mid (X, D) \text{~is lc, and ~} D \subset \Supp \Delta\}
\] is a rational polytope. By $\ka(K_X+\De; X/U) \geq 0$, there exists $m\in \Zz_{\geq 0}$ such that $\pi_*\Oo_{X}(\lf m(K_X+\De) \rf) \neq 0$, where $\pi: X \to U$. Because ${\lf m\De \rf}/{m}$ is a $\Qq$-divisor, there is a rational polytope $Q \subset R$ such that $\De \in Q$ and for any $D \in Q$, we have $\ka(K_X+D; X/U) \geq 0$. If $\De$ is a $\Qq$-divisor, then by the same argument for Theorem \ref{thm: main}, we reduce to the case where $K_Z+\Theta$ is $\Qq$-Cartier with klt singularities and big over $U$. Then the log canonical model exists by \cite[Corollary 1.1.5 (2)]{BCHM10}. If $\De$ is not a $\Qq$-divisor, then $\De$ lies in the relative interior of a face of $Q$. Replacing $Q$ by this face, we can assume that there is a rational polytope $P$ such that $P \subset Q^\circ$. By Theorem \ref{thm: main}, $(X, \De)$ has a log canonical model.
\end{proof}

\bibliographystyle{alpha}

\bibliography{bibfile}

\end{document}